\documentclass[preprint,12pt]{elsarticle}
\usepackage{amssymb}
\usepackage{amsthm}
\usepackage{amsmath}
\usepackage{epic}
\newtheorem{thm}{Theorem}[section]
\newtheorem{cor}[thm]{Corollary}
\newtheorem{lem}[thm]{Lemma}

\newtheorem{defi}[thm]{Definition}

\journal{arXiv}

\begin{document}

\begin{frontmatter}

\title{Characterizations of the spectral radius of nonnegative weakly irreducible tensors via digraph}
\author[label3]{Lizhu Sun}\ead{sunlizhu678876@126.com}
\author[label3]{Baodong Zheng}\ead{zbd@hit.edu.cn}
\author[label2]{Yimin Wei}\ead{ymwei@fudan.edu.cn}
\author[label1]{Changjiang Bu}\ead{buchangjiang@hrbeu.edu.cn}
\address[label3]{School of Science, Harbin Institute of Technology, Harbin 150001, PR China}
\address[label2]{School of Mathematical Sciences and Shanghai Key Laboratory of Contemporary Applied Mathematics, Fudan University, Shanghai, 200433, PR China}
\address[label1]{College of Science, Harbin Engineering University, Harbin 150001, PR China}
\begin{abstract}
For a nonnegative weakly irreducible tensor $\mathcal{A}$, we give some characterizations of the spectral radius of $\mathcal{A}$, by using the digraph of tensors. As applications, some bounds on the spectral radius of the adjacency tensor and the signless Laplacian tensor of  the $k$-uniform hypergraphs are shown.
\end{abstract}


\end{frontmatter}

\section{Introduction}
\label{Introduction }

An order $m$ dimension $n$ tensor $\mathcal{A}$ consist of $n^{m}$ complex entries,
$$\mathcal{A}=\left( {a_{i_1i_2\cdots i_m } } \right),~{1 \leqslant i_j\leqslant n}~ \left( {j = 1,2,\ldots ,m} \right).$$
The tensor $A=(a_{i_1i_2\cdots i_m})$ is called symmetric if $a_{i_1i_2\cdots i_m}=a_{\sigma(i_1)\sigma(i_2)\cdots \sigma(i_m)}$, where $\sigma$ is any permutation of the indices.
The tensor $\mathcal{A}$ is called nonnegative if all the entries $a_{i_1\ldots i_m } \geqslant 0$. Let $\mathbb{R}^{[m,n]}_+$ be the set of order $m$ dimension $n$ nonnegative tensors. And let $\mathbb{R}^{n}_{++}$ be the set of the dimension $n$ positive vectors (all the entries positive).

In 2005, the \textit{eigenvalue of tensors} is defined by Qi \cite{Qi2005} and Lim \cite{Lim 2005}, respectively.
 For a complex tensor $\mathcal{A}$ of order $m$ dimension $n$, a complex number $\lambda $ and a nonzero complex vector $x=(x_1,x_2,\ldots,x_n)^{\rm T}$ is called an eigenvalue and an eigenvector (corresponding to $\lambda$) of the tensor $\mathcal{A}$, respectively, if they satisfy
\[
\mathcal{A}x^{m - 1}  = \lambda x^{\left[ {m - 1} \right]},
\]
where $\mathcal{A}x^{m - 1}$ is a dimension $n$ vector with entry
 \[
\left( {\mathcal{A}x^{m - 1} } \right)_i  = \sum\limits_{i_2 ,...,i_m  = 1}^n {a_{ii_2 ...i_m } x_{i_2 } ...x_{i_m } },
i=1,2,\ldots,n,\]
and $x^{\left[ {m - 1} \right]}  = \left( {x_1^{m - 1} ,x_2^{m - 1},\ldots,x_n^{m - 1} } \right)^\mathrm{T}$ (see \cite{Qi2005}).
Let $\rho(\mathcal{A})=\sup\{|\lambda|:\lambda \in  {\rm spec}(\mathcal{A})\}$ be the spectral radius of tensor $\mathcal{A}$, where ${\rm spec}(\mathcal{A})$ is the set of all the eigenvalues of $\mathcal{A}$.

Recently, the spectral theory of tensors  has attracted much attention \cite{Bu2015, Chang, Chang 2011, Yang 2010, Yang 2013}.
In 2005, Lim \cite{Lim 2005} proposes the definition of \textit{irreducible tensors}.
In 2008, Chang et al. \cite{Chang} give the Perron-Frobenius Theorem for nonnegative irreducible tensors.
 It is shown that $\rho(\mathcal{A})$ is an eigenvalue of nonnegative irreducible tensors, and $\rho(\mathcal{A})$ is the only eigenvalue with nonnegative eigenvectors \cite{Chang}.
Similarly as the \textit{Collatz-Wielandt Theorem} of matrices, the \textit{Minimax Theorem} of the spectral radius of nonnegative irreducible tensors $\mathcal{A}$ is given as follows (see \cite{Chang})
\[
\min \limits_{x \in \mathbb{R}_{++} ^n }\max \limits_{1 \leqslant i \leqslant n} \frac{{\left( {\mathcal{A}x^{m - 1} } \right)_i }}
{{x_i^{m - 1} }} = \rho \left( \mathcal{A} \right) = \max \limits_{x \in \mathbb{R}_ {++} ^n} \min \limits_{1 \leqslant i \leqslant n} \frac{{\left( {\mathcal{A}x^{m - 1} } \right)_i }}
{{x_i^{m - 1} }}
.\]
Scholars pay much attention to find the largest eigenvalue of nonnegative irreducible tensors \cite{Chen2013, Ng, Ni2014, Zhou 2013}.
The \textit{weakly irreducible tensors} are defined by associated with tensors a digraph \cite{SF 2013, Pearson2014}.
\begin{defi}\label{def1}
An order $m$ dimension $n$ tensor $\mathcal{A}$ over real field is called weakly irreducible if the digraph $G_{\mathcal{A}}$ is strongly connected.
\end{defi}
Digraph $G_{\mathcal{A}}$ and the strongly connectivity are introduced in Section 2.
If tensor $\mathcal{A}$ is irreducible then it is weakly irreducible \cite{SF 2013, Pearson2014}. For the nonnegative weakly irreducible tensor $\mathcal{A}$, there exists a positive eigenvector corresponding to the eigenvalue $\rho(\mathcal{A})$ (see \cite{SF 2013}).

A hypergraph $\mathcal{H}=(V(\mathcal{H}), E(\mathcal{H}))$ is called \textit{$k$-uniform} if each edge of $\mathcal{H}$ contains exactly $k$ distinct vertices \cite{Cooper2012}.
The \textit{adjacency tensor} of $\mathcal{H}$, denoted by $\mathcal{A}_{\mathcal{H}}=(a_{i_1i_2\cdots i_k})$, is an order $k$
dimension $|V(\mathcal{H})|$ tensor with entries
\[
a_{i_1i_2\cdots i_k}=
\left\{ {\begin{array}{*{20}{c}}
   \frac{1}{(k-1)!}, ~ \{i_1i_2\cdots i_k\}\in E(\mathcal{H}), \\
   0 ,~~~~~~~\mbox{otherwises} .~~~~~ \\
\end{array}} \right.\]

The \textit{degree tensor} of the $k$-uniform uniform hypergraph $\mathcal{H}$, denoted by $\mathcal{D}_{\mathcal{H}}$, is an order $k$ dimension $|V(\mathcal{H})|$ diagonal tensor whose $(i\cdots i)$-diagonal entry is the degree of vertex $i$, $i=1,2,\ldots,|V(\mathcal{H})|$.
The tensor $\mathcal{Q}_{\mathcal{H}}=\mathcal{D}_{\mathcal{H}}+\mathcal{A}_{\mathcal{H}}$ is called the \textit{signless Laplacian tensor} of the $\mathcal{H}$.
Recently, the spectral theory of hypergraphs developed rapidly \cite{Cooper2012, Hu2015, Qi2014, Shao2015, Zhou2014}.

It is well-known that the nonnegative irreducible matrices are closely related to the digraphs \cite{RA1991}.
For a nonnegative irreducible matrices $M$ with all diagonal entries zero, Brualdi \cite{RA1982} gives the characterizations of the spectral radius of $M$ by using the associated digraph.

In this paper, we use the digraph of tensors to characterize the spectral radius of the nonnegative weakly irreducible tensors, which generalize the results of matrices to tensors \cite{RA1982}. By applying the characterizations, some bounds on the spectral radius of the adjacency tensor and signless Laplacian tensor of a $k$-uniform hypergraph are shown.

\section{Preliminary}

For an order $m$ dimension $n$ tensor $A=(a_{i_1i_2\cdots i_m})$, let $G_\mathcal{A}  = \left({V(G_\mathcal{A} ),E( G_\mathcal{A} )} \right)$ be the  digraph of the tensor $\mathcal{A }$ with vertex set
 $V (G_\mathcal{A} )= \left\{ {1,2,\ldots,n} \right\}$  and arc set $E( G_\mathcal{A} ) = \{ (i,j)| a_{ii_2 ...i_m }  \neq 0,j \in
\{ i_2 ,...,i_m  \}\}$ (see \cite{SF 2013, Pearson2014}).
If there exist directed paths from $i$ to $j$ and $j$ to $i$ for each $i, ~j\in V(G_\mathcal{A} )$ ($i\neq j$), then $ G_\mathcal{A}$ is called strongly connected.
Denote the set of the circuits in $G_\mathcal{A} $ by $C( G_\mathcal{A} )$ (Loops in the circuits are allowed).
Let $G_\mathcal{A}  ^ +  \left( v \right): = \left\{ {u\in V(G_\mathcal{A} ):\left( {v,u} \right) \in E( G_\mathcal{A} )} \right\}$.

Define a map $f$ from the vertex set of $G_\mathcal{A} $ to the real field, $f:V(G_\mathcal{A} )\rightarrow \mathbb{R}$, $f$ is called a  vertex labelling of $G_\mathcal{A} $.
By the Lemma 2.6 of \cite{RA1982}, we can get the following result.

\begin{lem}{\rm \cite{RA1982}}\label{lem1}
Let $G_\mathcal{A} $ be the digraph of $\mathcal{ A}$ with a vertex labelling $f$ on $V(G_\mathcal{A} )$. If $G_\mathcal{A} ^ +  \left( v \right)$ is nonempty for each $v \in V(G_\mathcal{A} )$, then there exist circuits $\{v_{i_1 },v_{i_2 } ,...,v_{i_k } ,v_{i_{k + 1} } \\= v_{i_1 }\}$ and $\{v_{t_1 },v_{t_2 } ,...,v_{t_s } ,v_{t_{s + 1} }  = v_{t_1} \}$ (Loops in the circuits are allowed) such that $f(v_{i_{j + 1} })=\max\{f(v):v\in G_\mathcal{A} ^ + ( v_{i_j} )\}$ and $f(v_{t_{l + 1} })=\min\{f(v):v\in G_\mathcal{A}^ + ( v_{t_l} )\}$ $j = 1,2,\ldots,k$, $ l = 1,2,\ldots,s$, respectively.
\end{lem}

\begin{lem}{\rm \cite{SF 2013}}\label{lem3}
Let $\mathcal{A}\in \mathbb{R}^{[m,n]}_+$ be a weakly irreducible tensor. Then $\rho(\mathcal{A})$ is an eigenvalue of $\mathcal{A}$, and
there exists a unique positive eigenvector corresponding to $\rho(\mathcal{A})$ up to a multiplicative constant.
\end{lem}

\begin{lem} {\rm \cite{shao 2013}}\label{lem5}
For an order $m$ dimension $n$ tensor $\mathcal{A}={(a_{i_1i_2\ldots i_m})}$ and an invertible diagonal matrix $D={\rm diag}(d_{11}, d_{22},\ldots, d_{nn})$, $\mathcal{B}=D^{-(m-1)}\mathcal{A}D$ is an order $m$ dimension $n$ tensor with entries
\[
b_{i_1i_2\ldots i_m}= d_{i_1 i_1}^{-(m-1)}a_{i_1i_2\ldots i_m}d_{i_2i_2}\cdots d_{i_mi_m}.
\]
In this case, $\mathcal{A}$ and $\mathcal{B}$ are called diagonal similar, and $\mathcal{A}$ and $\mathcal{B}$ have the same spectrum.
\end{lem}

\begin{lem}\label{lem6}{\rm \cite{Pearson2014}}
Let $\mathcal{H}$ be an $k$-uniform hypergraph. Then $\mathcal{A}_{\mathcal{H}}\in \mathbb{R}^{[m,n]}_+$ (and $\mathcal{Q}_{\mathcal{H}}\in \mathbb{R}^{[m,n]}_+$) is weakly irreducible if and only if $\mathcal{H}$ is connected.

\end{lem}
\section{Main results}

For a tensor $\mathcal{A}=\left( {a_{i_1 i_2 \cdots i_m } } \right)\in \mathbb{R}^{[m,n]}_+$, we denote the sum of $i$-th slice of $\mathcal{A}$ by $K_i= \sum\limits^n_{i_2,\ldots,i_m =1}{a_{ii_2 \cdots i_m } } $, $i=1,2,\ldots,n$. Let $|\gamma|$ be the length of the circuit $\gamma\in C (G_\mathcal{A} )$.

We first give two results on the bounds of spectral radius for nonnegative weakly irreducible tensors, which extend Theorem 4.7 and Corollary 4.6, 4.8 of \cite{RA1982} to tensors.

\begin{thm}\label{thm2}
Let $\mathcal{A} = \left( {a_{i_1 i_2 ...i_m } } \right) \in \mathbb{R}^{[m,n]}_+$ be a weakly irreducible tensor. Then
\[
\min_{\gamma \in C(G_\mathcal{ A})}\left(\prod\limits_{i\in \gamma } {K_i }\right)^{\frac{1}{|\gamma|}}\leqslant \rho (\mathcal{A})   \leqslant \max_{\gamma \in C(G_\mathcal{ A})} \left(\prod\limits_{i\in \gamma } {K_i }\right)^{\frac{1}{|\gamma|}}.
\]
\end{thm}
\begin{proof}
By Lemma \ref{lem3}, we suppose that
$x = \left( {x_1 ,x_2, \ldots,x_n } \right)^{\rm T}$ is a positive eigenvector corresponding to the eigenvalue $\rho(\mathcal{A})$. Let $f(i)=x_i$ ($i \in V(G_\mathcal{A})$) be the vertex labelling of digraph $G_\mathcal{A}$. From Definition \ref{def1}, we have $G_\mathcal{A}^+(i)$ is nonempty for each $i\in V(G_\mathcal{A})$.
 Lemma \ref{lem1} gives that there exists at least one circuit $\gamma_1  = \{i_1 ,i_2,\ldots,i_p ,i_{p + 1}  = i_1\}$ such that
 $x_{i_{j + 1} }  \geqslant x_k $, for each $ k \in G_\mathcal{A} ^ +  \left( {i_j } \right)$, $j = 1,2,\ldots,p$. Hence,
\begin{align*}
   \rho(\mathcal{A}) x_{i_j }^{m - 1} & = \sum\limits_{ k_2 ,\cdots,k_m =1}^n  a_{i_j k_2 ...k_m } x_{k_2} \cdots x_{k_m }  \hfill \\
   &\leqslant \left( {\sum\limits_{ k_2 ,\cdots,k_m =1}^n  {a_{i_j k_2 ...k_m } }} \right)x_{i_{j + 1} }^{m - 1}  \hfill \\
  & = K_{i_j } x_{i_{j + 1} }^{m - 1}  \hfill ,
\end{align*}
for $j = 1,2,\ldots,p$. Thus
\[
\left(\rho (\mathcal{A})\right)^p \prod\limits_{j = 1}^p x_{i_j }^{m - 1}   \leqslant \prod\limits_{j = 1}^p {K_{i_j } x_{i_{j + 1} }^{m - 1} }
.\]
Note that $x = \left( {x_1 ,x_1,\ldots,x_n } \right)^{\rm T}$ is positive, so we can get \[
\rho (\mathcal{A})   \leqslant \left(\prod\limits_{j = 1}^p {K_{i_j } }\right)^{\frac{1}{p}}
,\]
that is
\[
\rho (\mathcal{A})   \leqslant \left(\prod\limits_{i\in \gamma_1 } {K_i }\right)^{\frac{1}{p}}.
\]
Lemma \ref{lem1} also shows that there exists at least one circuit $\gamma_2=\{v_{t_1 } ,v_{t_2 },\ldots,v_{t_s } ,v_{t_{s + 1} }  = v_{t_1 }\}$ such that
 $x_{t_{l + 1} }  \leqslant x_k $ for each $ k \in G_\mathcal{A} ^ + ( t_l )$, $l = 1,2,\ldots,s$.
Similarly as the above proof, we can get
\[
\rho (\mathcal{A})  \geqslant \left(\prod\limits_{i\in \gamma_2 } {K_i }\right)^{\frac{1}{s}}.
\]
Thus
\[
\min_{\gamma \in C(G_\mathcal{ A})}\left(\prod\limits_{i\in \gamma } {K_i }\right)^{\frac{1}{|\gamma|}}
\leqslant \rho (\mathcal{A})   \leqslant \max_{\gamma \in C(G_\mathcal{ A})} \left(\prod\limits_{i\in \gamma } {K_i }\right)^{\frac{1}{|\gamma|}}.
\]
\end{proof}
\textbf{Remark.} It is easy to see that if $K_1=K_2=\cdots=K_n$, the equalities in the above theorem hold.

Lemma \ref{lem6} gives that the adjacency tensor (and signless Laplacian tensor) of a connected hypergraph $\mathcal{A}_{\mathcal{H}}$ (and $\mathcal{Q}_{\mathcal{H}}$) is nonnegative weakly irreducible, so $\rho(\mathcal{A}_{\mathcal{H}})$ (and $\rho(\mathcal{Q}_{\mathcal{H}})$) is an eigenvalue of $\mathcal{A}_{\mathcal{H}}$ (and $\mathcal{Q}_{\mathcal{H}}$).
In \cite{Cooper2012}, it is shown that the largest eigenvalue $\lambda_{\max}$ of $\mathcal{A}_{\mathcal{H}}$ is between the maximum degree $d_{\max}$ and the average degree $\overline{d}$ of $\mathcal{H}$, $\overline{d} \leqslant\lambda_{\max}\leqslant d_{\max} $. By using Theorem \ref{thm2}, we also show the bounds on the largest eigenvalue of adjacency tensor and signless Laplacian tensor in terms of the degrees.
\begin{thm}\label{thm5}
Let $\mathcal{H}$ be a connected $k$-uniform hypergraph with $n$ vertices. Suppose that $d_i$ is the degree of vertex $i$, $i=1,2,\ldots,n$. Then
\[
\min_{\gamma \in C(G_{\mathcal{ A}_{\mathcal{H}}})}\left(\prod\limits_{i\in \gamma } {d_i }\right)^{\frac{1}{|\gamma|}}\leqslant \rho(\mathcal{A}_{\mathcal{H}})   \leqslant \max_{\gamma \in C(_{\mathcal{ A}_{\mathcal{H}}})} \left(\prod\limits_{i\in \gamma } {d_i }\right)^{\frac{1}{|\gamma|}},
\]
and
\[
\min_{\gamma \in C(G_{\mathcal{ Q}_{\mathcal{H}}})}\left(\prod\limits_{i\in \gamma } {2d_i }\right)^{\frac{1}{|\gamma|}}\leqslant \rho(\mathcal{Q}_{\mathcal{H}})   \leqslant \max_{\gamma \in C(_{\mathcal{ Q}_{\mathcal{H}}})} \left(\prod\limits_{i\in \gamma } {2d_i }\right)^{\frac{1}{|\gamma|}}.
\]
\end{thm}

The shortest length of the circuits in $G_\mathcal{ A}$ is called the \textit{girth} of $G_\mathcal{ A}$.
If we bring an order to the slice sums $K_i$ $(i=1,2,\ldots,n)$ of tensor $\mathcal{A}$, the following result can be obtained.

\begin{thm}\label{thm1}
Let $\mathcal{A} = \left( {a_{i_1 i_2 ...i_m } } \right)\in \mathbb{R}^{[m,n]}_+ $ be a weakly irreducible tensor.  Suppose that $K_1  \leqslant K_2\leqslant \cdot\cdot\cdot \leqslant K_n$ and the girth of $G_\mathcal{ A}$ is $g$. Then
\[
\left ({K_1K_2\cdots K_g }\right)^{\frac{1}{g}}\leqslant  \rho \left( \mathcal{A} \right) \leqslant \left({K_{n - g + 1} K_{n - g + 2}\cdots K_n }\right)^{\frac{1}{g}}
.\]
\end{thm}
\begin{proof}
By Theorem \ref{thm2}, we have there exist circuits $\gamma_1,\gamma_2\in C(G_\mathcal{ A})$ of length $|\gamma_1|\geqslant g,~|\gamma_2|\geqslant g$ such that
\[
\left({\prod\limits_{i \in \gamma_2 } {K_i } }\right)^{\frac{1}{|\gamma_2|}}\leqslant \rho \left( \mathcal{A} \right) \leqslant \left({\prod\limits_{i \in \gamma_1 } {K_i } }\right)^{\frac{1}{|\gamma_1|}}
.\]
Since
\[
\left({\prod\limits_{i \in \gamma_1 } {K_i } }\right)^{\frac{1}{|\gamma_1|}} \leqslant \left({K_{n - |\gamma_1| + 1} K_{n - |\gamma_1| + 2}\cdots K_n }\right)^{\frac{1}{|\gamma_1|}} \leqslant \left({K_{n - g + 1} K_{n - g + 2}\cdots  K_n }\right)^{\frac{1}{g}}
\]
and
\[
\left({\prod\limits_{i \in \gamma_2 } {K_i } }\right)^{\frac{1}{|\gamma_2|}}\geqslant \left(K_1K_2\cdots K_{|\gamma_2|}   \right)^{\frac{1}{|\gamma_2|}}\geqslant \left(K_1K_2\cdots K_g   \right)^{\frac{1}{g}},
\]
we obtain the theorem holds.
\end{proof}

For a connected $k$-uniform hypergraph $\mathcal{H}$, According to the definition of hypergraph,
 there are no loops in $G_{\mathcal{A}_{\mathcal{H}}}$. And since $\mathcal{A}_{\mathcal{H}}$ is symmetric tensor,
 we have the girth of $G_{\mathcal{A}_{\mathcal{H}}}$ is $2$.
 Clearly, the girth of $G_{\mathcal{Q}_{\mathcal{H}}}$ is $1$.
Hence, it follows from Theorem \ref{thm5} and \ref{thm1} that the follow result can be obtained, which are shown in \cite{Cooper2012} and \cite{Hu2015+}, respectively.

\begin{cor} For a connected $k$-uniform hypergraph $\mathcal{H}$,
if $\mathcal{H}$ is regular with degree $d$, then $\rho(\mathcal{A}_{\mathcal{H}})=d$ and $\rho(\mathcal{Q}_{\mathcal{H}})=2d$.
\end{cor}

As we introduce in the first section, there are some results to find the spectral radius (largest eigenvalue) of nonnegative irreducible tensors. Here, we also give two theorems on the  minimum and the maximum characterizations of the spectral radius of nonnegative weakly irreducible tensors.

\begin{thm}\label{thm3}
Let $\mathcal{A} = \left( {a_{i_1 i_2 ...i_m } } \right) \in \mathbb{R}^{[m,n]}_+$ be a weakly irreducible tensor.  Then
\[
\min_{x\in \mathbb{R}_{++}^n} \max_{\gamma\in C(G_\mathcal{ A})}\left(\prod_{i\in\gamma}  \frac{(\mathcal{A}x^{m-1})_i}{x_i^{m-1}}\right)^{\frac{1}{|\gamma|}}=\rho(\mathcal{A})=\max_{x\in \mathbb{R}_{++}^n} \min_{\gamma\in C(G_\mathcal{ A})}\left(\prod_{i\in\gamma}  \frac{(\mathcal{A}x^{m-1})_i}{x_i^{m-1}}\right)^{\frac{1}{|\gamma|}}.
\]

\end{thm}

\begin{proof}
Let $\mathcal{B}=X^{-(m-1)}\mathcal{A}X$, where the matrix $X={\rm diag}(x_1,x_2,\ldots,x_n)$, $x_i>0$, $i=1,2,\ldots,n$. It is easy to see $\mathcal{B}$ is nonnegative weakly irreducible, and $G_\mathcal{ A}$ and $G_\mathcal{ B}$ are the same digraph. So by Theorem \ref{thm2}, we have
$$\min_{\gamma\in C(G_\mathcal{ A})}\left(\prod_{i\in\gamma}  K_i(\mathcal{B})\right)^{\frac{1}{|\gamma|}}\leqslant \rho(\mathcal{B})\leqslant\max_{\gamma\in C(G_\mathcal{ A})}\left(\prod_{i\in\gamma}  K_i(\mathcal{B})\right)^{\frac{1}{|\gamma|}}.$$
Calculation gives that
\[
 K_i(\mathcal{B})= \frac{1}{{x_i^{m - 1} }} \sum\limits^n_{ i_2,\cdots,i_m =1 } { {a_{ii_2 ...i_m } } x_{i_2 } ...x_{i_m } }=\frac{(\mathcal{A}x^{m-1})_i}{x_i^{m-1}},
\]
and since Lemma \ref{lem5} shows that $\rho(\mathcal{A})=\rho(\mathcal{B})$, we obtain
$$  \min_{\gamma\in C(G_\mathcal{ A})}\left(\prod_{i\in\gamma}
\frac{(\mathcal{A}x^{m-1})_i}{x_i^{m-1}}\right)^{\frac{1}{|\gamma|}}\leqslant  \rho(\mathcal{A})\leqslant \max_{\gamma\in C(G_\mathcal{ A})}\left(\prod_{i\in\gamma}
\frac{(\mathcal{A}x^{m-1})_i}{x_i^{m-1}}\right)^{\frac{1}{|\gamma|}}.$$
When $(x_1,x_2,\ldots, x_n)^{\rm T}$ is a positive eigenvector of $\mathcal{A}$ corresponding to $\rho(\mathcal{A})$, the equalities in the above equation hold, so we can get
\[
\min_{x\in \mathbb{R}_{++}^n} \max_{\gamma\in C(G_\mathcal{ A})}\left(\prod_{i\in\gamma}  \frac{(\mathcal{A}x^{m-1})_i}{x_i^{m-1}}\right)^{\frac{1}{|\gamma|}}=\rho(\mathcal{A})=\max_{x\in \mathbb{R}_{++}^n} \min_{\gamma\in C(G_\mathcal{ A})}\left(\prod_{i\in\gamma}  \frac{(\mathcal{A}x^{m-1})_i}{x_i^{m-1}}\right)^{\frac{1}{|\gamma|}}.
\]

\end{proof}
By the proof of the above theorem, we get the following result.

\begin{thm}
Let $\mathcal{A} = \left( {a_{i_1 i_2 ...i_m } } \right) \in \mathbb{R}^{[m,n]}_+$ be a weakly irreducible tensor.  Then
$$  \min_{\gamma\in C(G_\mathcal{ A})}\left(\prod_{i\in\gamma}
\frac{(\mathcal{A}x^{m-1})_i}{x_i^{m-1}}\right)^{\frac{1}{|\gamma|}}\leqslant  \rho(\mathcal{A})\leqslant \max_{\gamma\in C(G_\mathcal{ A})}\left(\prod_{i\in\gamma}
\frac{(\mathcal{A}x^{m-1})_i}{x_i^{m-1}}\right)^{\frac{1}{|\gamma|}},$$
where $x=(x_1,x_2,\ldots,x_n)$ is a positive vector.
\end{thm}

From Theorem \ref{thm1}, the following result can be obtained.

\begin{thm}\label{thm4}
Let $\mathcal{A}\in \mathbb{R}^{[m,n]}_+$ be a weakly irreducible tensor. Let the girth of $G_\mathcal{ A}$ is $g$. Then
$$\min_{X\in D^n}  \left({\prod_{i=n-g+1}^n K_i(X^{-(m-1)}\mathcal{A}X)}\right)^{\frac{1}{g}}=\rho(\mathcal{A}) =\max_{X\in D^n}  \left({\prod_{i=1}^gK_i(X^{-(m-1)}\mathcal{A}X)}\right)^{\frac{1}{g}},$$
where the slice sums of tensor $X^{-(m-1)}\mathcal{A}X$ are in the order $K_1  \leqslant K_2 \leqslant\cdot\cdot\cdot \leqslant K_n$ and $D^n$ is the set of all the $n\times n$ positive diagonal matrices.
\end{thm}

\textbf{Remark.}  Brualdi gives the characterizations of the spectral radius of a nonnegative irreducible matrices with all diagonal entries zero ( Corollary 4.10 and 4.11 of \cite{RA1982}). Theorem \ref{thm3} and \ref{thm4} generalize them to general nonnegative weakly irreducible tensors without the condition that diagonal entries are zero.

\end{document}